\documentclass[a4paper,12pt]{article}

\usepackage[latin1]{inputenc}
\usepackage[T1]{fontenc}
\usepackage[english]{babel}

\usepackage{mathrsfs}
\usepackage{amscd}
\usepackage{amsfonts}
\usepackage{amsmath}
\usepackage{amssymb}
\usepackage{amstext}
\usepackage{amsthm}
\usepackage{amsbsy}

\usepackage{xspace}
\usepackage[all]{xy}
\usepackage{graphicx}
\usepackage{tikz}
\usepackage{url}
\usepackage{latexsym}

\usepackage{graphicx} % support the \includegraphics command and options

\usepackage{booktabs} % for much better looking tables
\usepackage{array} % for better arrays (eg matrices) in maths
\usepackage{paralist} % very flexible & customisable lists (eg. enumerate/itemize, etc.)
\usepackage{verbatim} % adds environment for commenting out blocks of text & for better verbatim
\usepackage{subfig} % make it possible to include more than one captioned figure/table in a single float
\usepackage{fancyhdr} % This should be set AFTER setting up the page geometry
\usepackage{sectsty}
\allsectionsfont{\sffamily\mdseries\upshape} % (See the fntguide.pdf for font help)
\usepackage[nottoc,notlof,notlot]{tocbibind} % Put the bibliography in the ToC
\usepackage[titles,subfigure]{tocloft} % Alter the style of the Table of Contents

\usepackage[textwidth=100pt,textsize=footnotesize,bordercolor=white,color=blue!30]{todonotes}
\usepackage{hyperref} 

\makeatletter
\newcommand*{\rom}[1]{\expandafter\@slowromancap\romannumeral #1@}
\makeatother

\theoremstyle{definition}

\newtheorem{fact}{fact}

\newtheorem{prop}[fact]{Proposition}

\newtheorem{defini}[fact]{Definition}

\title{A Note on OTM-Realizability and Constructive Set Theories}
\author{Merlin Carl}

\begin{document}

\maketitle

\begin{abstract}
We define an ordinalized version of Kleene's realizability interpretation of intuitionistic logic by replacing Turing machines with Koepke's ordinal Turing machines (OTMs), thus obtaining a notion of realizability applying to arbitrary statements in the language of set theory. We observe that every instance of the axioms of intuitionistic first-order logic are OTM-realizable and consider the question which axioms of Friedman's Intuitionistic Set Theory (IZF) and Aczel's Constructive Set Theory (CZF) are OTM-realizable.

This is an introductory note, and proofs are mostly only sketched or omitted altogether. It will soon be replaced by a more elaborate version.
\end{abstract}

\section{Introduction}

Notions of effectivity, as appearing e.g. in the study of reverse mathematics or Weihrauch reducibility, are usually based on Turing computability and therefore restricted to objects that are either countable or allow for a countable encoding. Recently, there has been some interest in more general notions of effectivity, see e.g. \cite{EffMathUnc} or the much older \cite{EffField}. In \cite{GenEffRed} and \cite{OTMWeihrauch}, generalizations of Weihrauch reducibility were considered that are based on Koepke's ordinal Turing machines \cite{KoOTM} rather than finite Turing machines; further developments in this direction can be found in \cite{OTMRed}. 

Weihrauch reducibility is a notion of (relative) effectivity for $\Pi_{2}$-statements. It then becomes natural to ask for a notion of effectivity that applies to arbitrary set-theoretical statements. In this note, we propose such a notion by generalizing Kleene's realizability interpretation for intuitionistic arithmetic to general set theory via Koepke's Ordinal Turing Machines (OTMs) and explore its connection to two prominent systems of constructive set theory, namely Friedman's IZF \cite{IZF} and Aczel's CZF \cite{CSTBook}.\footnote{Myhill's CST is currently omitted.}

\section{Realizability with OTMs}

We now transfer Kleene's notion of realizability (see \cite{KlRealizability}) to Ordinal Turing Machines. For an account of Ordinal Turing Machines, see Koepke \cite{KoOTM}.

To apply OTMs, which can only process sets of ordinals, to arbitrary sets, we need to encode arbitrary sets as sets of ordinals. The standard technique to do this is the following:
Given a set $x$, form its transitive closure $\text{tc}(x)$, pick an suitable ordinal $\alpha$ and a bijection $f:\alpha\rightarrow x\cup\{\text{tc}(x)$ mapping $0$ to $x$ and let 
$c_{f}(x)=\{p(\iota,\xi):\iota,\xi<\alpha\wedge f(\iota)\in f(\xi)\}$, where $p$ denotes Cantor's pairing function on ordinals.

We can thus use OTMs to compute on on sets by giving codes of sets as the input and having them produce codes for sets as their output. This yields the following notion of OTM-realizability:

\begin{defini}{\label{otm realizable}}
   Let $\phi,\psi$ be $\in$-formulas, and let $P$ be an OTM-program, $\alpha\in\text{On}$, $a_{0},...,a_{n},b_{0},...,b_{m}$ sets with codes $c(a_{0}),...,c(a_{n}),c(b_{0}),...,c(b_{m})$ and $R$, $R^{\prime}$ be finite tuples. In the following, $P$ will always denote an OTM-program and $\alpha$ will denote an ordinal.

\begin{enumerate}
	\item If $\phi$ is quantifier-free, then $(P,\alpha)$ realizes $\phi(a_{0},...,a_{n})$ if and only if $\phi(a_{0},...,a_{n})$ is true (in any transitive sets containing $a_{0},...,a_{n}$). 
	\item $(R,R^{\prime})$ realizes $(\phi(a_{0},...,a_{n})\wedge\psi(b_{0},...,b_{m}))$ if and only if $R$ realizes $\phi(a_{0},...,a_{n})$ and $R^{\prime}$ realizes $\psi(b_{0},...,b_{m})$. 
	\item $(i,R)$ realizes $(\phi(a_{0},...,a_{n})\vee\psi(b_{0},...,b_{m}))$ if and only if $i=0$ and $R$ realizes $\phi(a_{0},...,a_{n})$ or $i=1$ and $R$ realizes $\psi(b_{0},...,b_{m})$. 
	\item $(P,\alpha)$ realizes $A\rightarrow B$ if and only if $P$ is an OTM-program and $\alpha$ is an ordinal such that, whenever a (code for a) realizer $R$ for $A$ is given as an input, $P(R,\alpha)$ computes a realizer $R^{\prime}$ for $B$.
	\item A realizer for $\neg\phi$ is a realizer for $\phi\rightarrow 1=0$.
	\item $(P,\alpha)$ realizes $\exists{x}\phi(x,a_{0},...,a_{n})$ if and only if $P(\alpha,c(a_{0}),...,c(a_{n}))$ halts with output $(c(b),R)$ where $c(b)$ codes a set $b$ such that 
	$R$ realizes $\phi(b,a_{0},...,a_{n})$. 
	\item $(P,\alpha)$ realizes $\forall{x}\phi(x,a_{0},...,a_{n})$ if and only if, for every code $c(a)$ for a set $a$, $P(\alpha,c(a),c(a_{0}),...,c(a_{n}))$ halts with output $R$ such that 
	$R$ realizes $\phi(a,a_{0},...,a_{n})$. 
	\item When $\phi$ contains the free variables $x_{1},...,x_{n}$, then $R$ realizes $\phi$ if and only if $R$ realizes $\forall{x_{1},...,x_{n}}\phi$.
\end{enumerate}
If there is $R$ such that $R$ realizes $\phi$, $\phi$ is called OTM-realizable.
\end{defini}

When evaluating the OTM-realizability of set-theoretical axioms, we will also need to consider axiom schemes like comprehension of replacement. Our interpretation is as follows: Let $A(\phi_{1},...,\phi_{n})$ be a formula in which $\phi_{1},...,\phi_{n}$ occur
as (propositional) variables. Then we say that $A$ is OTM-realizable if and only if there are an OTM-program $P$ and an ordinal $\alpha$ such that, whenever $i_{1},...,i_{n}$ is an $n$-tuple of G\"odel numbers for formulas, $P((i_{1},...,i_{n}),\alpha)$ computes a realizer for $A(\phi_{i_{1}},...,\phi_{i_{n}})$.

In the quantifier rules, we could also demand that the output of the programs $P$ does not depend on the encoding of the input $(a_{0},...,a_{n})$, but only on the input itself.
This is a stronger demand, which we call `absolute OTM-realizability'. Absolute OTM-realizability differs from OTM-realizability in some respects; for example, under sufficiently large cardinals, the statement $\forall{x\in\mathfrak{P}(\mathfrak{P}(\omega))\setminus\{\emptyset\}}\exists{z\in x}(z=z)$ is OTM-realizable, but not absolutely OTM-realizable.\footnote{We briefly sketch the argument for the interested reader: Suppose that $\mathfrak{P}(\mathfrak{P}(\omega))$ is large enough so that $L(\mathfrak{P}(\mathfrak{P}(\omega)))$ does not contain a choice function for $\mathfrak{P}(\mathfrak{P}(\omega))$. Suppose moreover that the forcing for making $\mathfrak{P}(\mathfrak{P}(\omega))$ well-ordered has two mutually generic filters $G$ and $F$ over $L(\mathfrak{P}(\mathfrak{P}(\omega)))$. The absolute realizability of the statement in question would imply the existence of a program $P$ computing a choice function $c$ for $\mathfrak{P}(\mathfrak{P}(\omega))$, independently of the coding of the input. Thus $c$ would be contained in 
	$L(\mathfrak{P}(\mathfrak{P}(\omega))))[G]\cap (L(\mathfrak{P}(\mathfrak{P}(\omega))))[F]=L(\mathfrak{P}(\mathfrak{P}(\omega)))$, a contradiction.} We will not discuss absolute OTM-realizability further in this note.

In the quantifier clauses of the above definition, we have allowed ordinal parameters. By the obvious modification, we can also
obtain a parameter-free version of (absolute) OTM-realizability. In this note, we only consider the version with parameters, even though the parameter-free version is arguably the more `constructive' one in spirit. This does not effect the results, below, though the details concerning the comprehension axiom are a bit different.

\subsection{Basic Observations}

We make some immediate observations on OTM-realizability.

For an $\in$-formula $\phi$, let $\bar{\phi}$ denote the classical negation of $\phi$ put in a form in which negation signs only appear in front of atomic formulas.

\begin{prop}
There are no $\in$-formula $\phi$ and sets $a_{0},...,a_{n}$ such that $\phi(a_{0},...,a_{n})$ and $\bar{\phi}(a_{0},...,a_{n})$ are both OTM-realizable.
\end{prop}
\begin{proof}
	An easy induction on syntax.	
\end{proof}

It might seem that OTM-realizability is a notion of `effective truth' that is stronger than mere truth, so that in particular, if $\phi$ is OTM-realizable, then $\phi$ will hold in the classical sense. This does indeed hold by an easy induction on formulas for formulas in negation normal form that do not contain implications. For statements using negation, however, it is false:

\begin{prop}
	There are $\phi$ and $\psi$ such that $\phi\rightarrow\psi$ is OTM-realizable, but false in $V$.
\end{prop}
\begin{proof}
	We will see below that the power set axiom POT is not OTM-realizable; however, it is clearly holds in $V$. Now POT$\rightarrow(1=0)$ is trivially OTM-realizable (as there are no OTM-realizers of POT, any OTM-program turns such an OTM-realizer into one of $1=0$), but also clearly false in $V$.
\end{proof}

\section{OTM-Realizability and Intuitionistic Provability}

We check that the deduction rules of the Hilbert-style calculus for intuitionistic logic preserve OTM-realizability.

The Hilbert calculus for intutionistic logic consists of $11$ logical axioms and four deduction rules. We need to show that (i) every instance of a logical axiom is OTM-realizable and (ii) that, whenever $\psi$ can be obtained by one of the deduction rules from $\phi$ and $\phi$ is OTM-realizable, then so is $\psi$.

We do not recall the logical axioms of intuitionistic Hilbert calculus here; they can be found in \cite{CSTBook}, Def. 2.5.1.

Furthermore, we have the following inference rules, see again \cite{CSTBook}, Def. 2.5.1:

\begin{enumerate}[i]
\item (Modus ponens) Given $\phi$ and $\phi\rightarrow\psi$, one may infer $\psi$
\item Given $\phi\rightarrow\psi$, where $x$ does not appear freely in $\phi$, one may infer $\phi\rightarrow\forall{x}\psi$
\item Given $\phi\rightarrow\psi$, where $x$ does not appear freely in $\psi$, one may infer $\exists{x}\phi\rightarrow\psi$.
\end{enumerate}

\begin{prop}
Every instantiation of one of the propositional axioms of intuitionstic logic by $\in$-formulas is OTM-realizable. Moreover, any instantiation of one of the quantifier axioms of intuitionistic logic by $\in$-formulas is OTM-realizable.
\end{prop}
\begin{proof}
The proof follows the classical proof that instances of the axioms are Kleene realizable. We consider an examplaric case.
Let $(\phi\rightarrow\psi)\rightarrow((\chi\rightarrow\psi)\rightarrow((\chi\rightarrow\psi)\rightarrow((\phi\vee\chi)\rightarrow\psi))$ be an instance of (7); we want to show that it is OTM-realizable, say by $R$.
That is, $R$ will have to turn a realizer $R_{(\phi\rightarrow\psi)}$ for $(\phi\rightarrow\psi)$ into a realizer for $((\chi\rightarrow\psi)\rightarrow((\phi\vee\chi)\rightarrow\psi))$.
Suppose that $R_{(\phi\rightarrow\psi)}$ is a realizer for $(\phi\rightarrow\psi)$, so $R_{(\phi\rightarrow\psi)}$ will turn realizers for $\phi$ into realizers for $\psi$.
We want to produce a program that turns realizers for $(\chi\rightarrow\psi)$ into realizers for $((\phi\vee\chi)\rightarrow\psi)$. Suppose a realizer $R_{(\chi\rightarrow\psi)}$ for $(\chi\rightarrow\psi)$ 
is given, and moreover, that we have a realizer $R_{(\phi\vee\chi)}$ for $(\phi\vee\chi)$. The latter is of the form $(i,R^{\prime})$, where $i\in\{0,1\}$ and $R^{\prime}$ is a realizer for $\phi$ when $i=0$ and a realizer for $\chi$ when $i=1$. The computation proceeds as follows: When $i=0$, apply $R_{(\phi\rightarrow\psi)}$ to $R^{\prime}$ to obtain a realizer for $\psi$. When $i=1$,
apply $R_{(\chi\rightarrow\psi)}$ to $R^{\prime}$ to again obtain a realizer for $\psi$. Clearly, this works as desired.
\end{proof}

\begin{prop}
When $\Phi$ is a set of OTM-realizable formulas and $\psi$ is obtained from elements of $\Phi$ via one of the deduction rules (i)-(iii), then $\psi$ is OTM-realizable.
\end{prop}
\begin{proof}
Again, this works as in the case of Kleene realizability.

The statement is clear for (i): If we have a program for turrning realizers for $\phi$ into realizers for $\psi$ and we have a realizer for $\phi$ available, we can obtain a realizer for $\psi$.
The first quantifier rule is immediate by the definition of realizability for formulas with free variables. For the second quantifier rule, note that the condition that $x$ does not appear freely in $\psi$
implies that $\phi\rightarrow\psi$ is OTM-realizable if and only if $\forall{x}(\phi\rightarrow\psi)$ is OTM-realizable; let $R$ be a realizer for this formula. Thus, for any $a$, $R$ will turn a realizer for $\phi(a)$ into a realizer for $\psi$. Now suppose that an OTM-realizer for $\exists{x}\phi$. This will consist of (a code for) some set $a$ and a realizer $R^{\prime}$ for $\phi(a)$. Applying $R$ to $R^{\prime}$ then yields the desired realizer for $\psi$.
\end{proof}

\section{Axioms and systems of constructive set theories}

We now discuss the OTM-realizability of the axioms of ZFC set theory and their most prominent constructive variants, as given in \cite{SEPIntST}.

It is easy to see that the axioms of Empty Set Existence, Extensionality, Pairing, Union and Infinity are OTM-realizable.

\begin{prop}
	The comprehension scheme $\forall{a}\exists{x}\forall{y}(y\in x\leftrightarrow (y\in a\wedge\phi(y))$ is not OTM-realizable.
	%has instantiations with $\in$-formulas $\phi$ that are not OTM-realizable. 
	However, the restriction of the comprehension scheme to $\Delta_{0}$-formulas is OTM-realizable.
	%every instantiation by a $\Delta_{0}$-formula is OTM-realizable.
\end{prop}
\begin{proof}
	Suppose for a contradiction that $(P,\alpha)$ realizes the comprehension scheme.
	Let $\phi(n,\alpha)$ be the formula expressing that the $n$th OTM-program halts in the parameter $\alpha$, and suppose that it has the G\"odel number $k$. Then $P(k,\alpha)$ would compute the halting problem for OTM-programs with parameter $\alpha$ using the parameter $\alpha$, which is clearly impossible.
	
	The second statement is an easy consequence of the ability of OTMs to evaluate $\Delta_{0}$-formulas, see \cite{KoOTM}.
\end{proof}

This is in agreement with the constructive criticism of unbounded comprehension as `impredicative'. However, by reflection in $L$, we have:

\begin{prop}
	If $V=L$, then every instantiation of the comprehension scheme is OTM-realizable.
\end{prop}

The idea is to pick, given a set $a$, a formula $\phi(x)$ and parameters $\vec{p}$, an $L$-level $L_{\alpha}$ containing $\vec{p}$ and $a$ reflecting $\phi$ and then to use $\alpha$ as a parameter. From a constructive viewpoint, this may be a reason to favor the parameter-free version of OTM-realizability.

\bigskip

The following may come as a small surprise; however, noting its dependence on the reading assigned here to implication, it is quite natural.

\begin{prop}
	The collection axiom $\forall{x\in X}\exists{y}\phi(x,y)\rightarrow\exists{Y}\forall{x\in X}\exists{y\in Y}\phi(x,y)$, and thus the replacement axiom and the strong collection axiom, are OTM-realizable.
\end{prop}
\begin{proof}
A realizer of $\forall{x\in X}\exists{y\in Y}\phi(x,y)$ consists in an OTM-program $P$ and an ordinal $\alpha$ such that, on input $c(a)$, $P(c(a),\alpha)$ will compute a code $c(b)$ for some set $b$ such that $\phi(a,b)$. But then, it suffices to apply this program to all elements of $X$ and collect the results together.
\end{proof}

\begin{prop}
	The power set axiom POT is not OTM-realizable, not even if one assumes $V=L$; moreover, the `subset collection scheme' from CZF has instantiations that are not OTM-realizable.
\end{prop}
\begin{proof}
Suppose for a contradiction that $(P,\alpha)$ realizes POT, i.e. the formula $\forall{x}\exists{y}\forall{z}(z\in y\leftrightarrow z\subseteq x)$. Then $P(\omega+\alpha,\alpha)$ computes the power-set of $\omega+\alpha\geq\alpha$. However, it is well known that OTM-computations cannot raise cardinals, see e.g. \cite{KoOTM} or \cite{GenEffRed}.

The subset collection scheme does not work either: For any infinite ordinal $\alpha$, take $a=\alpha$, $b=\alpha\times\{0,1\}$, and let $\phi(x,y,u)$ be the formula $\exists{i\in\{0,1\}}(y=(x,i)\wedge (i=1 \leftrightarrow x\in u))$. It is clearly trivial to realize $\forall{x \in\alpha}\exists{y\in b}\phi(x,y,u)$. Then the only set $d$ satisfying $\forall{\iota\in a}\exists{y \in b}\phi(\iota,y,u)\wedge\forall{y\in b}\exists{\iota\in a}\phi(\iota,y,u)$ is $\chi_{u\cap\alpha}$, the characteristical function of $u$.
  In our framework, $u$ can run through all sets, so a set $d$ that contains such a set $c$ for every $u$ must be a superset of $\{\chi_{s}:s\in\mathfrak{P}(\alpha)\}$. But this is again impossible because OTM-computations cannot raise cardinals. (But also note that the power set of $\alpha$ can easily be OTM-computed from a superset of the set of all characteristic function of elements of the powerset in the parameter $\alpha$.)

%$a=b=\alpha$, then let $u$ run through the non-empty subsets $w$ of $\alpha$ and let $\phi$ define an OTM-computable surjection from $\alpha$ to $w$ (which is easy: $f(\iota)=\iota$ when $\iota\in w$, otherwise $f(\iota)$ is the minimal element of $w$). This generates the power-set of $\alpha$, and so our counterexample from above works again.
\end{proof}

We note that the second part continues to hold even if the parameter $u$ is additionally restricted to ordinals.

\begin{prop}
The ``ordinal subset collection scheme'', which is obtained from the subset collection scheme by requiring the parameter $u$ to be an ordinal, has instantiations that are not OTM-realizable.
\end{prop}
\begin{proof}
We slightly modify the proof of the last proposition: Again let $a=\alpha$, $b=\alpha\times\{0,1\}$. For a given ordinal $\gamma$ (which plays the role of the parameter $u$), define $f(\gamma)$ by $$f(\gamma)=\begin{cases}\text{the }<_{L}\text{-minimal element of }\mathfrak{P}(\alpha)\cap L_{\gamma+1}\setminus L_{\gamma}\text{, if such a set exists}\\ \emptyset\text{, otherwise}\end{cases}$$.

It is easy to see that $f$ is OTM-computable. Let $\phi(x,y,\gamma)$ be the formula $\exists{i\in\{0,1\}}(y=(x,i)\wedge (i=1\leftrightarrow x\in f(\gamma)))$. Then the only set $d$ satisfying $\forall{\iota\in a}\exists{y \in b}\phi(\iota,y,u)\wedge\forall{y\in b}\exists{\iota\in a}\phi(\iota,y,u)$ is $\chi_{f(\gamma)\cap\alpha}$. Thus, a set $c$ containing such a $c$ for all ordinals $\gamma$ would be a superset of the set of all characteristical functions of constructible subsets of $\alpha$. From this, one could easily OTM-compute the constructible powerset of $\alpha$, which, however, is impossible by the last proposition.
\end{proof}

\begin{prop}
	The axiom of Regularity is OTM-realizable, but not absolutely so. The axiom of $\in$-induction is absolutely OTM-realizable.
\end{prop}
\begin{proof}	
\end{proof}

\begin{prop}
	The axiom of choice in the form $\forall{x}\exists{y}(\emptyset\notin x\rightarrow \forall{z\in x}(|y\cap z|=1))$ is not OTM-realizable. 
	
	The axiom of choice in the alternative formulation $\forall{x}(\forall{y\in x}\exists{z\in y}(z=z)\rightarrow\exists{f:x\rightarrow\bigcup{x}}\forall{y\in x}(f(y)\in y))$
	is OTM-realizable.
	
	The well-ordering principle WO is not OTM-realizable.
\end{prop}
\begin{proof}
	For the first and the last claim, see \cite{GenEffRed}. The second claim again is an easy consequence of the interpratation of implication.
\end{proof}

\bigskip
\textbf{Note}: Friedman's IZF has the standard axioms of ZF; saw that not all of them are OTM-realizable. Aczel's CZF has the axioms of extensionality, pairing, union, empty set, infinity, bounded separation, strong collection, the subset collection scheme and the axiom of regularity. Of these, only the subset collection scheme is not OTM-realizable. 
All axioms of Kripke-Platek set theory (KP) are OTM-realizable.

\section{Further Work}

We have only begun to explore connections between notions of generalized effectivity and constructive set theory. In particular, absolute OTM-realizability is not explored in this note. Moreover, one could base the realizability interpretation on other notions of generalized effectivity; for example, in \cite{EffField}, Hodges uses primitive recursive set functions, which has the advantage to avoid the intermediate encoding via ordinals. It would be interesting to see how stable the above results are under various changes of the underlying formalism. A somewhat more ambitous goal would be to axiomatize the (absolutely) OTM-realizable statements under intuitionistic provability. 
Moreover, we will explore in further work the connections between generalized effectivity and idealized agency of mathematics, such as advocated by P. Kitcher (\cite{NatMathKnow}).

\end{document}